\documentclass[11pt]{article}
\usepackage{amsmath, amsthm, amssymb, cancel, cite, geometry, graphicx, overpic, url, comment, todonotes}

\renewcommand{\theenumi}{(\roman{enumi})} 

\newtheorem{theorem}{Theorem}[section]

\newtheorem{corollary}[theorem]{Corollary}
\newtheorem{lemma}[theorem]{Lemma}

\DeclareMathOperator{\R}{\mathbb{R}}
\DeclareMathOperator{\Sym}{\mathbb{S}}
\DeclareMathOperator{\Z}{\mathbb{Z}}

\DeclareMathOperator{\interior}{int}
\DeclareMathOperator{\lspan}{span}
\DeclareMathOperator{\relint}{relint}
\DeclareMathOperator{\bd}{bd}

\newcommand{\iprod}[2]{\left\langle {#1}, {#2} \right\rangle}

\title{Everything is possible:\\
constructing spectrahedra with prescribed facial dimensions}
\author{Vera Roshchina\thanks{School of Mathematics and Statistics, UNSW Sydney, Kensington Campus, NSW 2052, Australia (e-mail: v.roshchina@unsw.edu.au). This author is grateful to the Australian Research Council for continuing support.} \and Levent Tun\c{c}el\thanks{Department of Combinatorics and Optimization, Faculty of Mathematics, University
of Waterloo, Waterloo, Ontario N2L 3G1, Canada (e-mail: levent.tuncel@uwaterloo.ca). Research of this
author was supported in part by Discovery Grants from the Natural Sciences and Engineering Research
Council (NSERC) of Canada.}}

\date{December 8, 2023, revised: August 22, 2024}
\begin{document}

\maketitle

\begin{abstract}
    Given any finite set of nonnegative integers, there exists a closed convex set whose facial dimension signature coincides with this set of integers, that is, the dimensions of its nonempty faces comprise exactly this set of integers. In this work, we show that such sets can be realised as solution sets of systems of finitely many convex quadratic inequalities, and hence are representable via second-order cone programming problems, and are, in particular, spectrahedral. It also follows that these sets are facially exposed, in contrast to earlier constructions. We obtain a lower bound on the minimum number of convex quadratic inequalities needed to represent a closed convex set with prescribed facial dimension signature, and show that our bound is tight for some special cases. Finally, we relate the question of finding efficient representations with indecomposability of integer sequences and other topics, and discuss a substantial number of open questions.

\end{abstract}

\vspace{0.5cm}

{\bf MSC:} 90C22, 90C25, 52A38



\section{Introduction}

Understanding the boundary structure of feasible regions of optimization problems is an essential ingredient in optimization theory (in determining existence and uniqueness of optimal solutions, in characterization of optimal solutions or subsets of optimal solutions, in design and analyses of efficient algorithms). In this paper, we address some of the fundamental questions about the dimensions of faces of closed convex sets in Euclidean spaces. Our treatment applies to the feasible regions of Second-Order Cone Programming problems and the feasible
regions of Semidefinite Programming problems. The latter are also called \emph{spectrahedra}. More precisely, a set $S$ in $\R^n$ is called a \emph{spectrahedron}, if there exists a positive integer $m$ and $m$-by-$m$ symmetric matrices $A_0, A_1, A_2, \ldots , A_n$ such that
\[
S=\left\{ x \in \R^n \,\, : \,\, A_1 x_1 +A_2 x_2 + \cdots + A_n x_n \succeq A_0 \right\},
\]
where for two symmetric matrices $A,B$ of the same size, $A \succeq B$ means, $(A-B)$ is positive semidefinite. We denote the space of 
$m$-by-$m$ symmetric matrices by $\Sym^m$.

It was shown in \cite{RoshYost} that for any finite set of nonnegative integers containing zero, there exists a compact convex set such that the dimensions of the faces of this set form exactly this prescribed set of integers. The particular way in which the compact sets satisfying this property were constructed in \cite{RoshYost} by means of Minkowski sums of Euclidean balls of different dimensions, results in sets that are not facially exposed. In this work, we propose a different construction that is not based on Minkowski sums and our construction possesses favourable properties, including facial exposedness and representation by convex quadratic inequalities. We also discuss the complexity of such representations, and pose a number of open questions. 

Boundary structure of spectrahedra has been studied for some time (see \cite{FriedlandLoewy1976,Barvinok1995,LaurentPoljak1996,Pataki2000} and the references therein). Among the rich behaviours of the boundary structure of spectrahedra, understood so far, are the behaviours similar to general convex sets in the context of strict complementarity failures~\cite{deCarliSilvaT2019}, and
the fact that every rank (in the range of possible ranks of an extreme point of a spectrahedron described in the results of \cite{FriedlandLoewy1976, Barvinok1995,Pataki2000}) is possible for an extreme point of a spectrahedron~\cite{Scheiderer2022}.
In this paper, we describe another context (dimensions of faces) in which a spectrahedron can show as rich a behaviour as any closed convex set.

Recall that a convex subset $F$ of a convex set $C$ is a \emph{face} of $C$, denoted by $F\unlhd C$, if for any $x,y\in C$ with $(x,y)\cap F \neq \emptyset$ we have $x,y\in F$. The dimension of a convex set is the dimension of the smallest affine subspace that contains this set. Since each face is a convex set, dimensions of nonempty faces are well-defined. Since the empty set is a face of every convex set, for consistency of notation it is sometimes prescribed the dimension of $-1$. In our context, empty faces are not directly relevant; hence, we only focus on the nonempty faces and their dimensions.

For a convex set $C\subseteq \R^n$ its \emph{facial dimension signature} is the set of nonnegative integers $I$ that consists of the dimensions of faces of $C$, that is, \[
I = \{\dim F\, :\, F\unlhd C, \, F\neq \emptyset\}.
\]
A face $F$ of a convex set $C$ is \emph{exposed} if there exists a closed half-space $H$ such that the set $C$ is a subset of this half-space, and the intersection of its boundary hyperplane with $C$ is exactly $F$. A convex set is \emph{facially exposed}, if every proper face of the convex set is exposed. Our first main result is that any facial dimension signature is realisable as the solution set of a finite system of convex quadratic inequalities. More precisely, we prove the following. 

\begin{theorem}[Everything is possible]\label{thm:everything} For every nonempty finite set $I$ of nonnegative integers, there exists a closed convex set $S\subseteq \R^d$ such that $d=\max I$, and $I$ is the facial dimension signature of $S$. 
More specifically, it is possible to construct this set $S$ in such a way that 
\begin{enumerate}
    \item $S$ is facially exposed;
    \item $S$ can be represented as the solution set of a system of $|I|-1$ convex quadratic inequalities;
    \item if $0\in I$, then in addition, $S$ can be assumed to be compact. 
\end{enumerate}
\end{theorem}
The proof of this theorem is constructive and is based on building the set $S$ as an intersection solution sets of convex quadratic inequalities. We explain the construction in Section~\ref{sec:construction} and prove Theorem~\ref{thm:everything} in Section~\ref{sec:proof}.

The explicit construction that we use to prove Theorem~\ref{thm:everything} is not always optimal, in the sense that there may exist a convex set with the same signature representable by fewer convex quadratic inequalities. We give an example of an optimal construction that only requires $\lceil\log_2 |I|\rceil$ inequalities for a `complete' signature $I = \{0,1,\dots,n\}$ and relate the constructive approach to efficiently realising other signatures with the beautiful topic of indecomposable representations of integer sequences. 

Our last main result is a lower bound on the number of convex quadratic inequalities needed to realise a given set $I$ of nonnegative integers as a facial dimension signature of a convex set.

\begin{theorem}\label{thm:lowerbound} Given a set of nonnegative integers $I$ the number of quadratic inequalities needed to represent a convex set with facial dimension signature $I$ can not be smaller than the minimum  number $k$ of integers $d_1\geq d_2\geq \cdots\geq d_k$ such that  
\begin{equation}\label{eq:intervals}
I\subseteq \{n\}\bigcup_{m=1}^k \left\{ i\in \Z_+\,:\, \sum_{j=1}^m d_j - (m-1)n\leq i\leq d_m\right\},   
\end{equation}
where $n = \max I$.
\end{theorem}
We prove this result in Section~\ref{sec:lowerbound}.
Our paper is organised as follows. In Section~\ref{sec:generalconstruction} we focus on the constructive proof of Theorem~\ref{thm:everything}. Within this section we first explain the construction that allows to build convex sets with desired facial dimension signatures in Section~\ref{sec:construction}, in Section~\ref{sec:convgeom} we recap some foundational facts from convex geometry that allow us to prove Theorem~\ref{thm:everything} in Section~\ref{sec:proof}. In  Section~\ref{sec:complexity} we study optimal representations and complexity. We begin with introducing a construction that gives an optimal representation of complete signatures in Section~\ref{sec:construction}, that only requires $\lceil\log_2 |I| \rceil$ inequalities for a `complete' signature $I = \{0,1,\dots,n\}$. In Section~\ref{sec:integersequences} we explore some possible ways of realising signatures with fewer inequalities than $|I|-1$, and relate this with indecomposable integer sequences. In Section~\ref{sec:lowerbound} we prove Theorem~\ref{thm:lowerbound}. Finally in Section~\ref{sec:conclusions} we discuss a range of open problems that emerged from this work.  

\section{General construction}\label{sec:generalconstruction}

The goal of this section is to prove Theorem~\ref{thm:everything}. We first introduce our key construction, and recall some preliminary facts from convex geometry. 

\subsection{Construction}\label{sec:construction}

The idea behind our construction lies in intersecting carefully selected convex sets to achieve the desired properties of their facial structure. To build some intuition, first note that complete facial dimension signatures are attained by simplices; i.e., for any nonnegative integer $n$, any $n$-dimensional simplex has the facial dimension signature of $\{0, 1, \ldots,n\}$. Also, any nonempty, pointed, $n$-dimensional polyhedron has the same signature. Using polyhedra that are not necessarily pointed, it is possible to obtain any set of contiguous nonnegative integers as a facial dimension signature (via taking a Minkowski sum of a suitable linear subspace and a suitable nonempty, pointed polyhedron). However, no other subset of nonnegative integers can be the facial dimension signature of a polyhedron.

Next, to continue building some more intuition, suppose that we would like to construct a closed convex set whose facial dimension signature is $\{0,2,3\}$. If we start with the Euclidean ball, we already have faces of dimensions $0$ and $3$, so it remains to graft a two-dimensional face onto this ball without changing the rest of the facial dimension signature. If we intersect this Euclidean ball with a closed half-space whose boundary plane intersects the ball through its interior, then the intersection of the boundary plane with the ball will generate a two-dimensional face (see the leftmost image in Fig.~\ref{fig:constIllustrate}).
\begin{figure}[ht]
    \centering
    \includegraphics[width=0.3\textwidth]{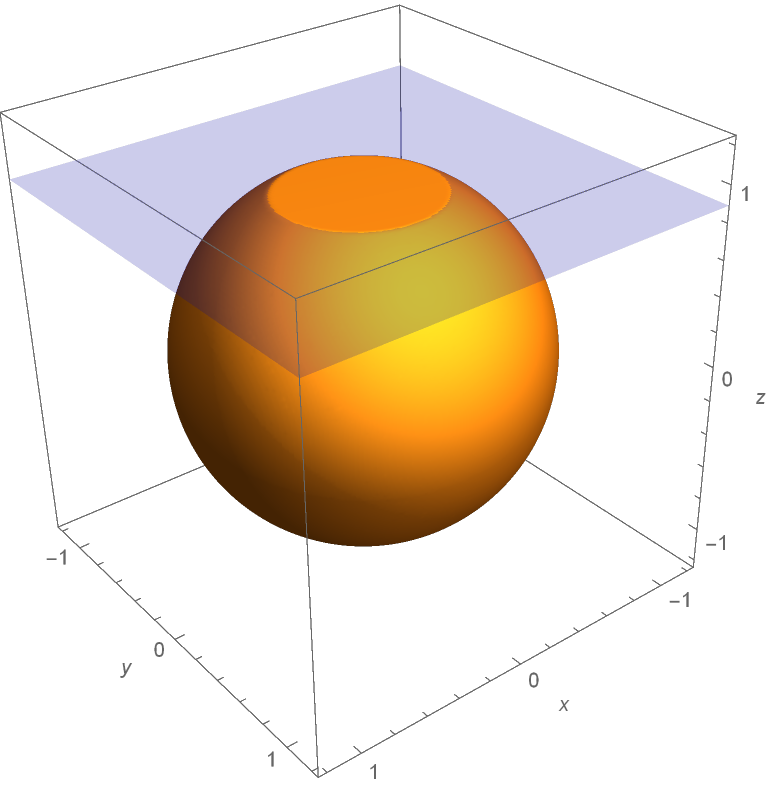}
    \quad 
    \includegraphics[width=0.3\textwidth]{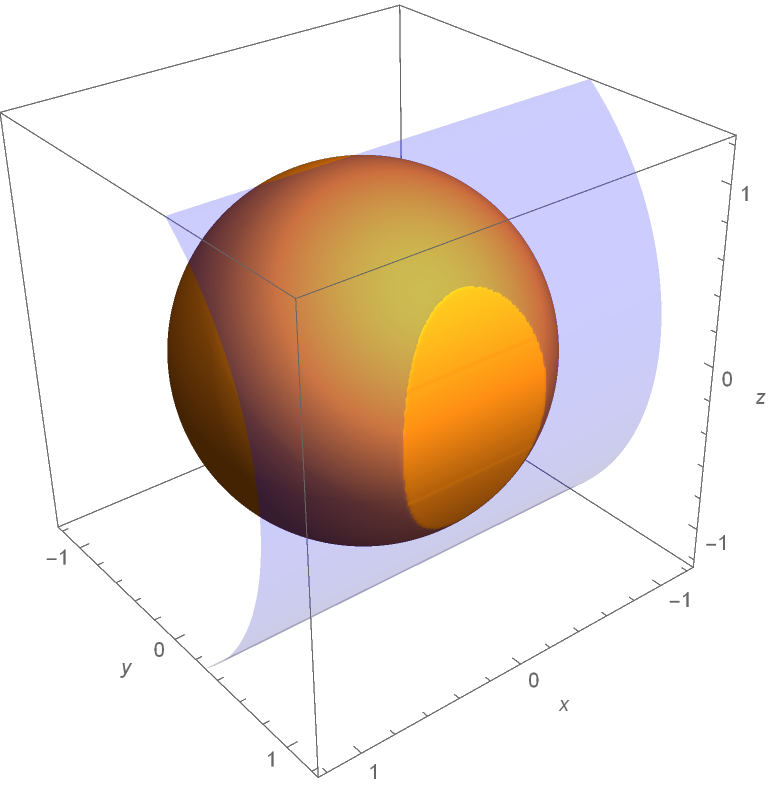}
    \quad 
    \includegraphics[width=0.3\textwidth]{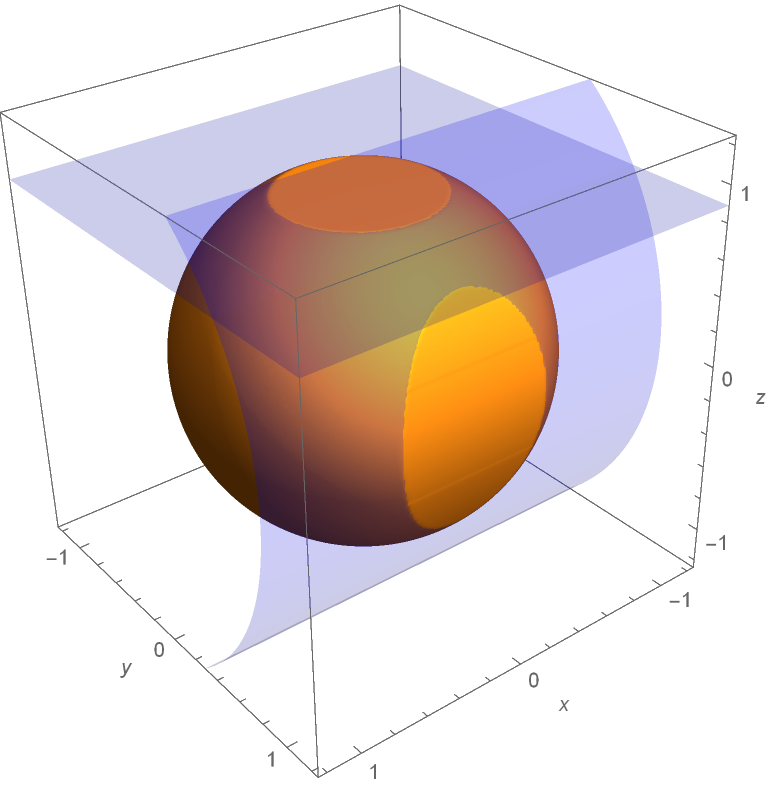}
    \caption{Construction of the set $S$ for when $I = \{0,2,3\}$, $\{0,1,3\}$ and $\{0,1,2,3\}$. Here $c = 1/\sqrt{2}$ and $r = 8/5$.}
    \label{fig:constIllustrate}
\end{figure}
Likewise, if we want to generate one-dimensional faces on our ball, instead of a half-space, we can intersect the ball with a cylinder, as shown in the middle image in Fig.~\ref{fig:constIllustrate}. Placing such objects strategically, we can generate faces of any dimension we like, as is shown in the rightmost image in Fig.~\ref{fig:constIllustrate} (also see Fig.~\ref{fig:faces} for a diagram which highlights the faces of this set that represent different dimensions).
\begin{figure}[ht]
    \centering
    \begin{overpic}[
    width=0.5\textwidth]{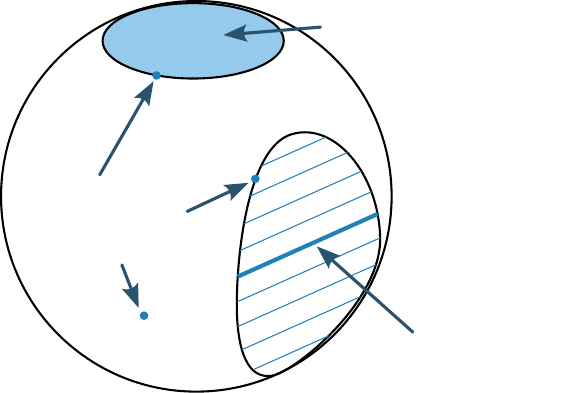}
    \put(60,64){face of dimension 2}
    \put(77,11){faces of}
    \put(73,6){dimension 1}
    \put(10,32){faces of}
    \put(6,27){dimension 0}
    \end{overpic}
    \caption{Faces of the set $S$ for $I = \{0,1,2,3\}$.}
    \label{fig:faces}
\end{figure}
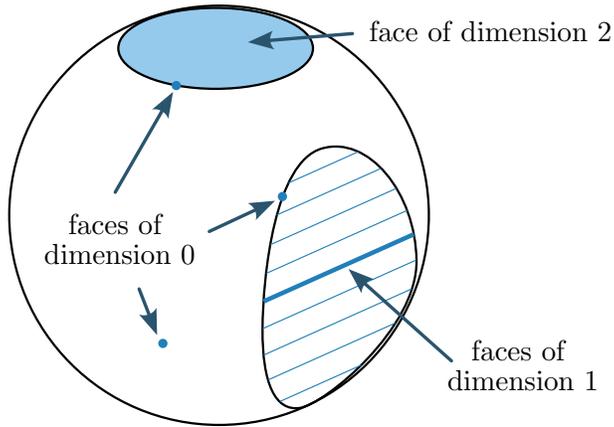 This construction can be  generalised to realise any subset of nonegative integers as a facial dimension signature of a convex set. 

Let $I$ be a finite set of nonnegative integers, and assume that $0 = \min I < \max I = n$. Let $c,r>0$ be such that 
\begin{equation}
    \label{eq:restrictcr}
1+c >r > \sqrt{c^2+ \sqrt{2} c +1}.
\end{equation}
Notice that $r>c$ and that such numbers exist, for instance, $c = 1/\sqrt{2}$ and $r = 8/5$. The allowable values of $r$ and $c$ are sketched in Fig.~\ref{fig:rc}.
\begin{figure}[ht]
    \centering
    \includegraphics[width=0.5\textwidth]{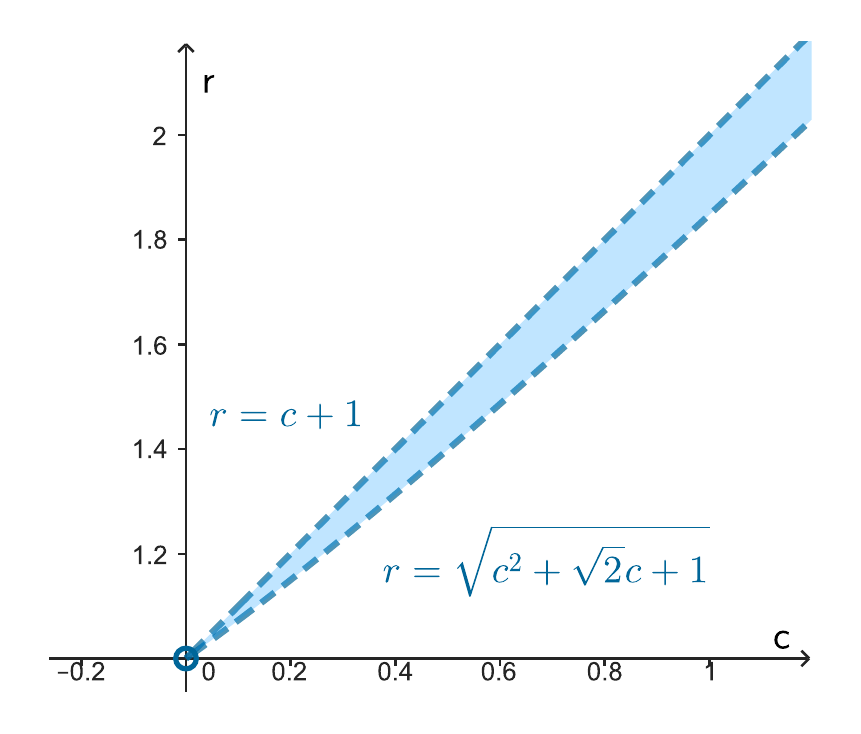}
    \caption{The region defined by the inequalities \eqref{eq:restrictcr}}
    \label{fig:rc}
\end{figure}

By $B^n$, denote the closed Euclidean ball of radius one in $\R^n$ centred at the origin. For any $i \in \{1,\dots, n-1\}$ by $C^n_i$ denote the following cylindrical subset of $\R^n$, 
\begin{equation}\label{eq:Cin}
C_i^n := \{x \in \R^n \, : \, (x_{i+1}+c)^2 + x_{i+2}^2+\cdots + x_n^2 \leq r^2\} =  \R^i \oplus (r B^{n-i} - c e_1),    
\end{equation}
where $e_1$ is the first standard basis vector (in $\R^{n-i}$).

Note that $C_i^n$ admits the following spectrahedral representation.
\[
C_i^n= \left\{ x\in \R^n \,\, : \,\, rI +\left(e_{i+1}e_{n+1}^{\top}+e_{n+1}e_{i+1}^{\top}\right)c + \sum_{\ell=i+1}^n \left(e_{\ell}e_{n+1}^{\top}+e_{n+1}e_{\ell}^{\top}\right) x_{\ell} \succeq 0\right\}.
\]
In the above, $I$ is the $(n+1)$-by-$(n+1)$ identity matrix and $e_i$ is the $i^{\textup{th}}$ standard basis vector (in $\R^{n+1}$).

Let $S_0:= B^n$, that is, $S_0$ is the unit Euclidean ball centred at zero. Then for any $i\in \{1,\dots, n-1\}$ we recursively define
\[
S_{i} := \begin{cases}
    S_{i-1} & \text{ if } i \notin I,\\
    S_{i-1}\cap C_i^n & \text{ if } i \in I.
\end{cases}
\]
We let $S = S_{n-1}$. This construction generates closed convex sets of the kind shown in Fig.~\ref{fig:constIllustrate}, and in Fig.~\ref{fig:slices} we show three-dimensional slices of the four-dimensional set $S$ that corresponds to $I = \{0,1,2,3,4\}$. 
\begin{figure}[ht]
    \centering
    \includegraphics[width=0.3\textwidth]{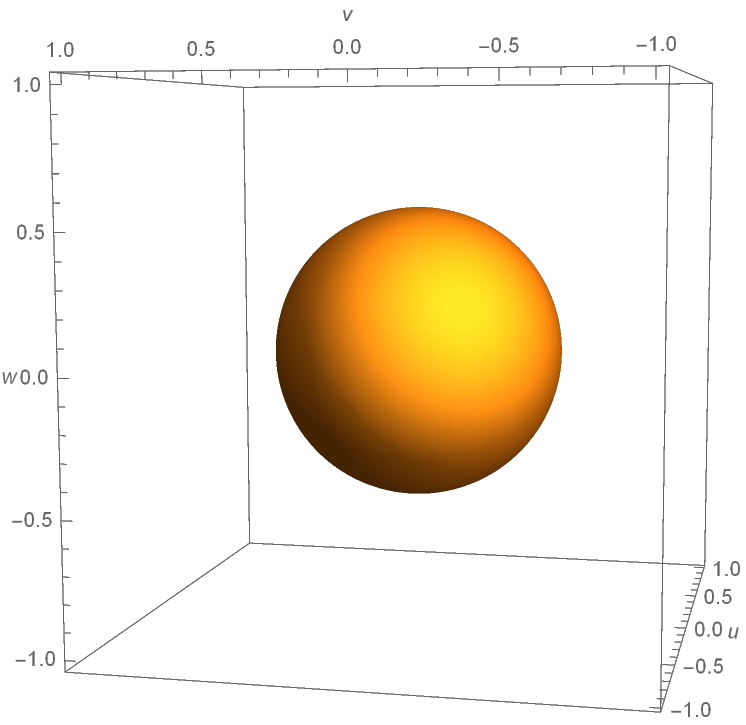}
    \quad 
    \includegraphics[width=0.3\textwidth]{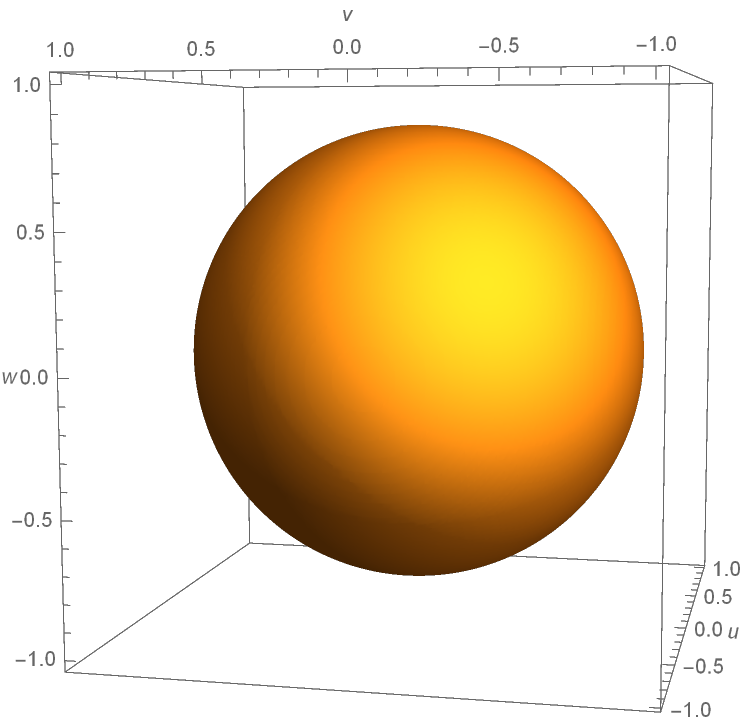}
    \quad 
    \includegraphics[width=0.3\textwidth]{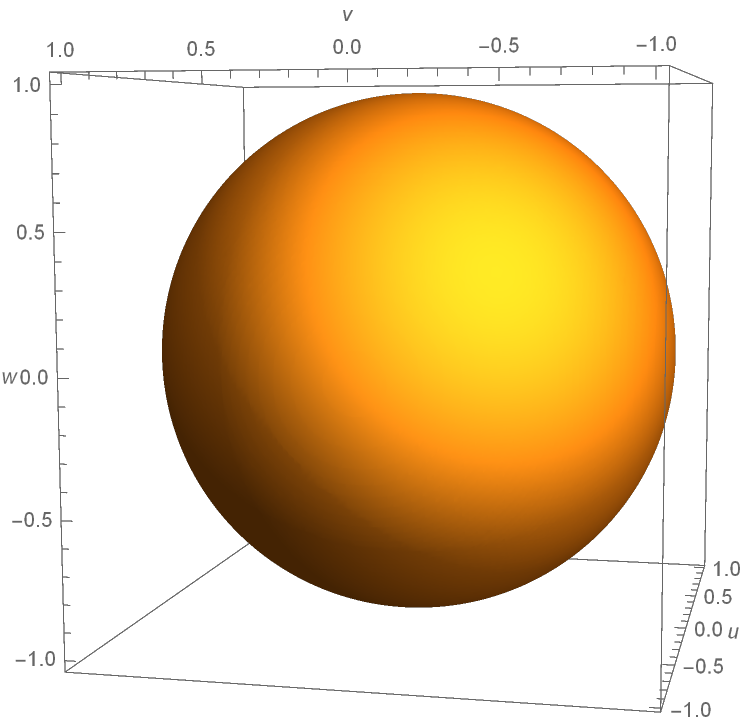}
    \quad 
    \includegraphics[width=0.3\textwidth]{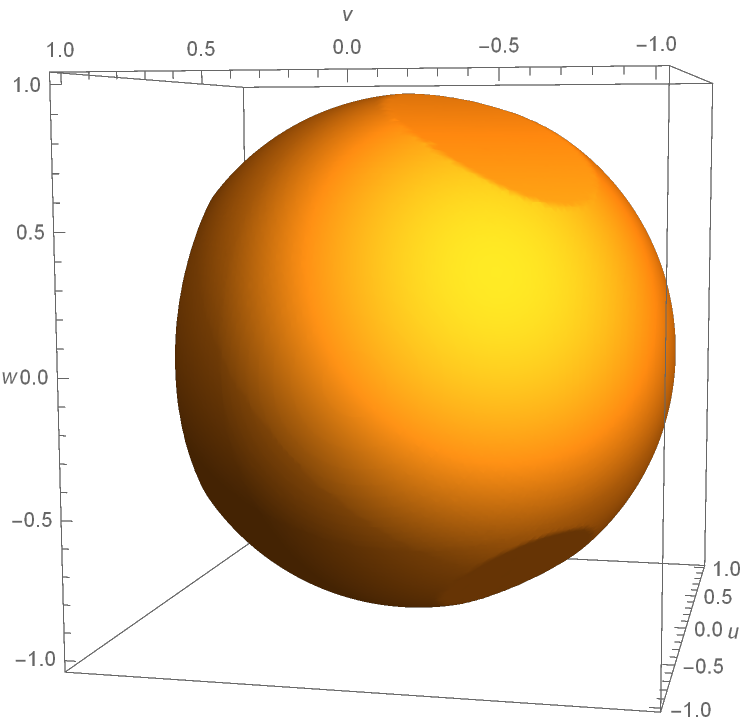}
    \quad 
    \includegraphics[width=0.3\textwidth]{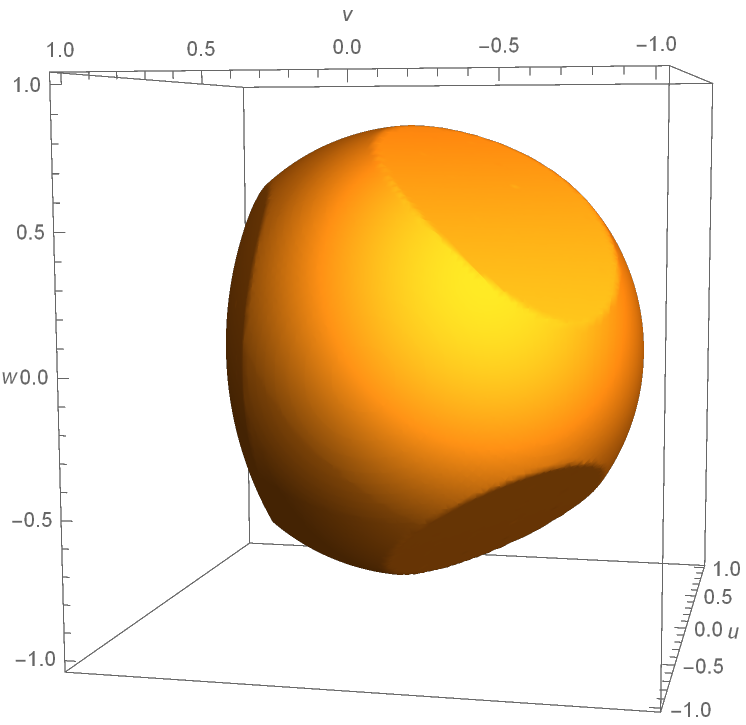}
    \quad 
    \includegraphics[width=0.3\textwidth]{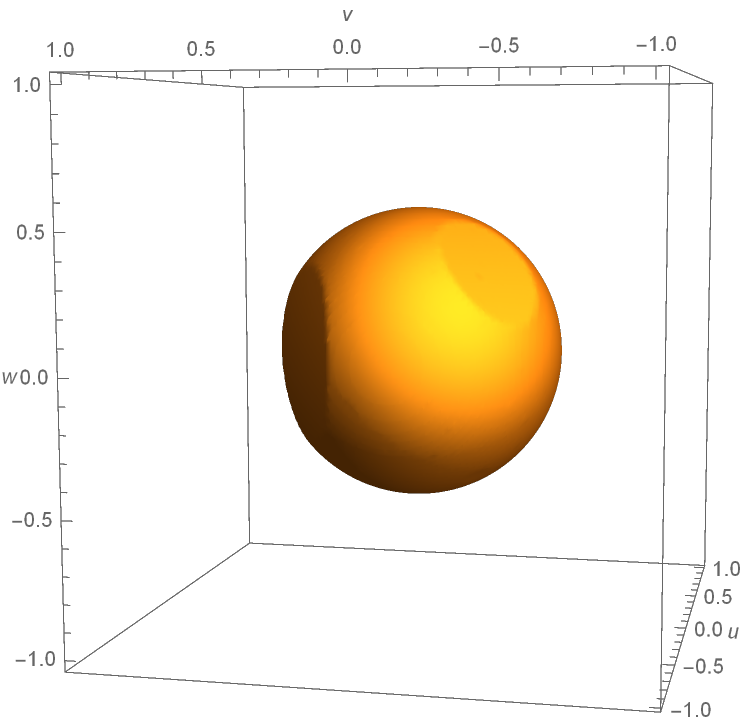}
    \quad 
    \caption{Slices of the set $S$ for $I = \{0,1,2,3,4\}$, by affine subspaces orthogonal to the vector $(1,1,1,1)$.}
    \label{fig:slices}
\end{figure}
In Section~\ref{sec:proof} we prove that this construction indeed realises any finite subset $I$ of nonnegative integers with $\min I =0$ as the facial dimension signature of a compact convex set. In the next section, we recall some technical tools and facts which will be used in the proofs.

\subsection{Essential facts from convex geometry}\label{sec:convgeom}

The dimension of a convex set $C\subseteq \R^n$ is the dimension of the smallest affine subspace which contains $C$. This smallest affine subspace is called the \emph{affine hull of} $C$. An affine subspace is a shifted linear space, $\{x\}+ L$, where $L$ is a linear subspace, and if for an affine subspace $A := \{x\} + L$, $C\subseteq A$, then we may assume that $x\in C$. Since intersection of any collection of linear subspaces is a linear subspace, the minimal affine subspace is well-defined (as the intersection of all affine subspaces containing $C$), along with its dimension.

The \emph{relative interior of $C$}, denoted by $\relint(C)$ is the interior of the set $C$ with respect to its affine hull. The relative interior of any nonempty convex set is also nonempty. 
Relative interior is closely related to the notion of \emph{minimal face}. For any $x\in C$ the \emph{minimal face of $C$ containing $x$}, $F_{\min}(x,C)$, can be defined as
\[
F_{\min}(x,C) := \bigcap_{\substack{F\unlhd C\\ x\in F} } F.
\]
Since intersections of faces are faces, $F_{\min}(x,C)$ is indeed the minimal face (by set inclusion) of $C$ that contains $x$. A basic fact about minimal faces is that $x\in \relint F$, where $F\unlhd C$ if and only if $F = F_{\min}(x,C)$ (see, for instance, Corollary~3.14 in \cite{DMR} for a proof). From this observation follows the classical decomposition of any closed convex set into the disjoint union of relative interiors of its faces (see \cite[Theorem~18.2]{Rockafellar}). The following result is Corollary~2.6 in \cite{DMR} and will be useful for our proofs. 

\begin{lemma}\label{lem:charminF} The minimal face $ F_{min}(x,C)$ for any $x\in C\subseteq\R^n$ can be represented as
\[
F_{\min}(x,C) = \bigcup \{[y, z] \subseteq C \,: \, x \in (y, z)\},
\]    
where we use the convention that $(x,x) = [x,x] =\{x\}$, that is, $(x,y)$ is the set of all strict convex combinations of $x$ and $y$.  
\end{lemma}

\begin{lemma}\label{lem:faces} Let $C,D\subseteq \R^n$ be convex. Then the nonempty faces of $C\cap D$ are exactly the nonempty intersections of faces of $C$ and $D$.     
\end{lemma}
\begin{proof} Let $F$ be a nonempty face of $C\cap D$. There exists an $x\in \relint F$, and so $F = F_{\min}(x,C\cap D)$. It follows from Lemma~\ref{lem:charminF} that 
$F = F_{\min}(x,C\cap D)\subseteq F_{\min}(x,C)\cap F_{\min}(x, D)$. Now let $y\in F_{\min}(x,C)\cap F_{\min}(x, D)$. By the same result, it must be possible to extend the segment $[x,y]$  beyond $x$ within both $C$ and $D$, and hence there exists some $z\in C\cap D$ such that $x\in (y,z)$. This means that $y\in F$, utilising the same characterisation again. We conclude that $ F_{\min}(x,C)\cap F_{\min}(x, D)\subseteq F_{\min}(x,C\cap D)$, and hence we have represented $F$ as an intersection of faces of $C$ and $D$. 

Conversely, let $E\unlhd C$ and $F\unlhd D$. Then for any $y,z \in C\cap D$ such that there is an $x\in E\cap F \cap (y,z)$ we have $y,z \in E \cap F$, and hence $E\cap F$ is a face of $C\cap D$. 

\end{proof}

\begin{lemma}\label{lem:prodfaces} Suppose that $C\subseteq \R^k$ and $D\subseteq \R^m$ are convex sets. Then the nonempty faces of the direct sum of $C$ and $D$ are exactly the direct sums of nonempty faces of $C$ and $D$; that is,
\[
\{F\, : \, F\unlhd (C\oplus D) \}  = \{E\, :\, E\unlhd C\}  \oplus \{G\, :\, G \unlhd D\}.
\] 
\end{lemma}
\begin{proof}
Consider $E\unlhd C$ and $G\unlhd  D$, both nonempty.  We will show that $(E\oplus G) \unlhd (C\oplus D)$. 

Suppose that  $x = (x_1,x_2)$, $y= (y_1,y_2)$ and  $z=(z_1,z_2)$ are such that  
$x\in (E\oplus G)$, $y,z\in (C\oplus D)$, and $x\in (y,z)$. Then $x_1 \in (y_1,z_1)$ and $x_2 \in (y_2,z_2)$. Since $E$ is a face of $C$, and $y_1,z_1\in C$, we conclude that $y_1,z_1 \in E$. Likewise $y_2,z_2\in G$. We conclude that $y,z\in (E\oplus G)$, and so by the definition of a face $E\oplus G$ is a face of $C\oplus D$. 

Conversely, assume that $F$ is a nonempty face of $C\oplus D$. We will show that $F = E\oplus G$ for some faces $E\unlhd C$ and $G\unlhd D$. 

Since $F$ is nonempty, it has nonempty relative interior. Then there exists some  $x = (x_1,x_2) \in \relint F$. Let $E = F_{\min}(x_1,C)$ and $G = F_{\min}(x_2,D)$.    We know from the first part of the proof that $E\oplus G $ is a face of $C\oplus D$. Since $x\in (E\oplus G)$, and $F$ is the minimal face of $C \oplus D$ containing $x$, we must have $F\subseteq (E\oplus G)$. Now take any $y \in (E\oplus G)$. Since $E$ and $G$ are the minimal faces of $x_1$ in  $C$ and $x_2$ in $D$ respectively, by Lemma~\ref{lem:charminF} there must be some $\alpha_1,\alpha_2>0$ such that 
\[
x_1+ \alpha_1 (x_1-y_1) \in E, \quad x_2+ \alpha_2 (x_2-y_2) \in G.
\]
Let $\alpha  = \min \{\alpha_1,\alpha_2\}$. By convexity,
\[
z_1 := x_1 + \alpha (x_1-y_1) \in E, \quad z_2 := x_2 + \alpha (x_2-y_2) \in G.
\]
Then $z := (z_1,z_2)  = x+ \alpha (y-x)\in E\oplus G$, and using Lemma~\ref{lem:charminF} again we conclude that $y\in F$. We have shown that $F\subseteq (E\oplus G)$, which together with the converse inclusion demonstrated earlier proves that $F = (E\oplus G)$. 
\end{proof}

\begin{lemma}\label{lem:prodR} Suppose that $C\subseteq \R^m$ is a nonempty convex set. Then for any positive integer $k$
\[
\dim (C\oplus \R^k) = \dim C + k.
\]    
\end{lemma}
\begin{proof} Evident from observing that the affine subspaces containing $(C\oplus \R^k)$ are the direct sums of affine subspaces containing $C$ with $\R^k$, and from the fact that the dimension of a direct sum of two convex sets is the sum of the dimensions of the individual convex sets.
\end{proof}

\begin{lemma}\label{lem:shift} Let $C\subseteq \R^n$ be a convex set, and let $I$ be its facial dimension signature. Then, for every $k\in \Z_+$ the set
\[
C\oplus \R^k
\]
has the facial dimension signature
\[
I+\{k\} = \{i+k\, : \, i\in I\}.
\]
\end{lemma}
\begin{proof} 
    It is evident that the only nonempty face of $\R^k$ is $\R^k$ itself, since it coincides with its interior (this in particular follows from the `faceless theorem' \cite[Theorem~2.7]{Faceless}). By Lemma~\ref{lem:prodfaces}, the nonempty faces of $C\oplus \R^k$ are exactly 
    \[
    \{F \oplus \R^k \, : \, F\unlhd C, \; F\neq \emptyset\}.
    \]
    By Lemma~\ref{lem:prodR}, for any nonempty face $F$ of $C$ we have $\dim (F\oplus \R^k) = \dim(F) + k$. We conclude that the facial dimension signature of $C+\R^k$ is exactly $I+\{k\}$. 
\end{proof}

\begin{lemma}\label{lem:intintersection} Suppose that $x$ is in the relative interior of a face $F$ of some convex set $C$, and also that $x$ is in the interior of some other convex set $D$. Then the intersection $E := F\cap D$ is a face of $C\cap D$ and $\dim E = \dim F$. 
\end{lemma}
\begin{proof}
    Since $D$ is its own face, and $x\in D\cap F$, the intersection $E = F\cap D$ is nonempty, and by Lemma~\ref{lem:faces} it is a face of $C\cap D$. It remains to show that $\dim E = \dim F$. Since $E\subseteq F$, we have $\dim E \leq \dim F$, and it remains to show the converse. To do so, it is sufficient to demonstrate that any affine subspace that contains $E$ also contains $F$. Suppose that $E\subseteq A = x+L$, where $L = \lspan \{l_1,\dots, l_k\}$, and assume  that there is some $y\in F\setminus A$. Since $x\in \interior C$, there must be a sufficiently small $\alpha\in (0,1)$ such that $y_\alpha = x+ \alpha (y-x)$ is in $C$. At the same time  by the convexity of $F$ we must have $y_\alpha \in F$, and therefore $y_\alpha \in E$. However, then $y_\alpha - x\in L$, and so $y-x\in L$, and we conclude that $y\in A$, a contradiction to our assumption. We conclude that the minimal affine subspace which contains $E$ must also contain $F$, and hence, $\dim E\geq \dim F$. 
\end{proof}

We will need a few technical statements about faces of sets $C_i^n$ and the Euclidean balls, before we can prove the main result.

\begin{lemma}\label{lem:balldimensions}
The nonempty faces of the Euclidean ball $B^n$ are the set $B^n$ itself, and all points on the boundary, each representing a face of dimension $0$. Hence, the facial dimension signature of $B^n$ is $\{0,n\}$.    
\end{lemma}
\begin{proof}
Since any convex set is represented as the disjoint union of relative interiors of its faces, to list all faces of a convex set it is sufficient to list all minimal faces of all points of this set.

If $x$ is an interior point of the Euclidean ball $B^n$, then the minimal face of $x$ is the ball itself. Since the ball has nonempty interior, its dimension is $n$. Every point $x$ on the boundary of $B^n$ is exposed by the half-space 
$\langle x,c\dot \rangle \leq 1$, hence every boundary point is a face of $B^n$, of dimension $\{0\}$. There are no other faces, since we have exhausted the points of $B^n$.  
\end{proof}

\begin{lemma}\label{lem:facedimensions} For any positive integer $n\geq 2$  and $i\in \{1,\dots, n-1\}$, the facial dimension signature of the set $C_{i}^n$ defined in \eqref{eq:Cin} is exactly $\{i, n\}$.  
\end{lemma}
\begin{proof}  
This statement follows from the geometric characterization of $C_{i}^n$ given in \eqref{eq:Cin} and Lemmas~\ref{lem:shift} and \ref{lem:balldimensions}.

\end{proof}

\begin{lemma}\label{lem:dimintersections} Suppose that $C\subseteq \R^n$ is a full-dimensional convex set. Let $I$ denote the facial dimension signature of $C$.  Then the facial dimension signature of the intersection $B^n\cap C$ is a subset of $I\cup \{0\}$. 
\end{lemma}
\begin{proof}
 By Lemma~\ref{lem:faces} the nonempty faces of the intersection of two convex sets are exactly the intersections of faces of these convex sets. By Lemma~\ref{lem:balldimensions} the nonempty faces of the Euclidean ball are the ball itself and all points on the boundary (each representing an extreme point). Hence the faces of  $B^n\cap C$ are either extreme points or the intersections of the faces of $C$ with the ball $B^n$. If a face $F$ has an interior intersection with the ball, then by Lemma~\ref{lem:intintersection}  the dimension of $F\cap B^n$ is the same as the dimension of $F$. Otherwise, since the relative interiors of $F$ and $C$ do not intersect, $F$ can be separated from $B^n$ by a hyperplane. This hyperplane is supporting $B^n$, and as such its intersection with $B^n$ is a face. This cannot be $B^n$ itself, hence by Lemma~\ref{lem:balldimensions} this intersection is a single extreme point. We conclude that the intersection of $F$ with $B^n$ is a singleton, and hence this face has dimension $0$. 
\end{proof}

\begin{lemma}\label{lem:2.7} Facial dimension signatures are invariant under affine isomorphisms of the space.
\end{lemma}
\begin{proof}
   Let $C \subseteq \R^n$ be a convex set and $\phi: \R^n \to \R^n$ be an affine isomorphism. Then, there exist $A \in \R^{n \times n}$ nonsingular and $b \in \R^n$ such that $\phi(x) = Ax + b$. Thus, $\phi(C) = \{ Ax + b \,\, : \,\, x \in C \}$. Since the dimensions of affine subspaces are preserved under affine isomorphisms, dimensions of affine hulls of every subset of $C$, including dimensions of affine hulls of its faces are preserved. Therefore, the facial dimension signature of $\phi(C)$ is the same as that of $C$.
   \end{proof}

We denote by $\Sym^n$ the set of $n$-by-$n$ symmetric matrices, and by $\Sym^n_+$ the set of positive semidefinite matrices in $\Sym^n$.

\begin{lemma}\label{lem:quadraticdim} A nonempty convex set $C\subseteq \R^n$ defined by a convex quadratic inequality 
\[
C := \{x\in \R^n\, : \, f(x) \leq 0\}, 
\]
where $f :\R^n \to \R$ is a convex quadratic function is \begin{itemize}
    \item 
either an affine space and hence has the facial dimension signature $\{m\}$ for some $m \in \{0, 1, \ldots, n \}$, 
\item
or, it has the facial dimension signature $\{m,n\}$,
where $m\in \{0,\dots,n-1\}$.
\end{itemize} 
\end{lemma}
\begin{proof} We may assume that $C$ is not empty and is not an affine space. Since $f$ is a convex quadratic function, there exist $A \in \Sym^n_+$, $a \in \R^n$ and $\alpha \in \R$ such that
$f(x) = \iprod{Ax}{x} + \iprod{a}{x} +\alpha$. Thus,
\[
C = \{x\in \R^n\, : \, \iprod{Ax}{x} + \iprod{a}{x} +\alpha\leq 0\}.
\]
Since $f$ is not a constant function (otherwise $C$ is an affine space), at least one of $A$, $a$ is nonzero. If $A$ is the zero matrix, then $a$ is not zero
and $C$ becomes a closed half-space with facial dimension signature $\{n-1,n\}$. Thus, we may assume, $A \neq 0$ for the rest of the proof. Let $m \in \{0,1, \ldots,n-1\}$ denote the dimension of the null space of $A$. Then, as we prove next, the facial dimension signature of $C$ is either $\{m,n\}$ or $m \geq 1$ and the facial dimension signature is $\{m-1,n\}$.

By Lemma~\ref{lem:2.7}, facial dimension signature is invariant under linear isomorphisms; therefore, the facial dimension signature of $C$ is the same as that of $C'$, where
\[
C' := \left\{x \in \R^n \, : \, \sum_{i=1}^{n-m} x_i^2 - \sum_{j=n-m+1}^{n-\ell} x_j + \alpha' \leq 0  \right\},
\]
for some $\ell \in \{0, 1, \ldots, m\}$ and $\alpha' \in \R$. (To construct a linear isomorphism which takes $C$ to $C'$, one can first apply a spectral decomposition to $A$ to construct a linear isomorphism which makes every degree two term equal to $x_i^2$ for some $i \in \{1,2, \ldots,n-m\}$. Then, for every $i$ for which the terms $x_i^2$ and $x_i$ are both present, one can express these terms as $(x_i-c_i)^2 + \bar{c}_i$ for some constants $c_i, \bar{c}_i \in \R$. Then, after a shift in $\R^n$ and a non-zero scaling of the variables, we obtain the form $C'$. Since each transformation above is a linear isomorphism of the space, so is their composition.)

Thus, we may assume that our set is $C'$.
If $m=0$, then the second sum is empty and since $C$ is not a singleton ($C$ is not an affine space), then $\alpha'$ is negative and $C'$ is an $n$-dimensional Euclidean ball with facial dimension signature $\{0,n\}$ by Lemma~\ref{lem:balldimensions}. Thus, we may further assume $m\geq 1$.

If $m \geq 1$ and the second sum in the definition of $C'$ is empty (i.e., $\ell =m$), then by Lemma~\ref{lem:shift} with $k:=n-m$ and Lemma~\ref{lem:balldimensions}, the facial dimension signature of $C'$ and also that of $C$ is $\{m,n\}$. Only remaining case is $m\geq 1$, $\ell \in \{0,1, \ldots,m-1\}$. In this case, the facial dimension signature of $C'$ as well as that of $C$ is $\{m-1,n\}.$ To see this, let $F$ be a proper nonempty face of $C'$. Further let $u \in F$ and $\lambda \in (0,1)$ such that for a distinct pair $x,y \in C'$, $u=\lambda x +(1-\lambda)y$. Then, using the facts that 
$x,y \in C'$, $\lambda \in (0,1)$, we obtain
\begin{eqnarray}
    \label{eqn:2.11.1}
    \lambda \sum_{i=1}^{n-m} x_i^2 
    +(1-\lambda) \sum_{i=1}^{n-m} y_i^2 +\alpha' &\leq &  \lambda \sum_{j=n-m+1}^{n-\ell} x_j + (1-\lambda) \sum_{j=n-m+1}^{n-\ell} y_j.
        \end{eqnarray}
By definition, the RHS is equal to $\sum_{j=n-m+1}^{n-\ell} u_j$. Since $u$ is on the boundary of $C'$, this is also equal to 
\[
\sum_{i=1}^{n-m} u_i^2 + \alpha'
=\lambda^2 \sum_{i=1}^{n-m} x_i^2 +(1-\lambda)^2 \sum_{i=1}^{n-m} y_i^2 +2\lambda(1-\lambda)\sum_{i=1}^{n-m} x_i y_i +\alpha'. \]
These last two equations we mentioned, together with the inequality \eqref{eqn:2.11.1}, and the fact that $\lambda \in (0,1)$ imply $x_i=y_i$ for every $i\in\{1,2, \ldots,n-m\}$; and, as a result, we also deduce
$\sum_{j=n-m+1}^{n-\ell} x_j =\sum_{j=n-m+1}^{n-\ell} y_j = \sum_{j=n-m+1}^{n-\ell} u_j$. Hence, $\dim(F) \leq m-1$. Noting that for a fixed $u$, every solution of the last linear system may be extended uniquely to an element of the face $F$ implies, $F$ contains an affine subspace of dimension $m-1$. Therefore, $\dim(F)=m-1$ and the facial dimension signature of $C'$ is $\{m-1,n\}$ as we had claimed.
\end{proof}

A careful reading of the above proof indicates that every facial dimension signature mentioned in the statement of
Lemma~\ref{lem:quadraticdim} can be realised by a suitable convex quadratic inequality. Therefore, Lemma~\ref{lem:quadraticdim} gives a complete characterization of facial dimension signatures of all nonempty convex sets which can be realised as the solution set of a single quadratic inequality.

\subsection{Proof of Theorem~\ref{thm:everything}}\label{sec:proof}

\begin{proof}
First observe that if $|I|=1$, then $S=\R^d$, where $\{d\}=I$ gives the required convex set, defined by an empty set of convex quadratic inequalities. In this special case when $I = \{0\}$, we have $S = \R^0 = \{0\}$, which is a compact set.

We now focus on the case when $|I|\geq 2$, and equivalently $\max I >\min I$, and we
will prove Theorem~\ref{thm:everything} for the case $0\in I$ first. Then we will explain how to modify this construction for a more general setting.

Given a finite set $I$ of nonnegative integers that contains zero, we use the construction described in Section~\ref{sec:construction} to obtain the set 
\begin{equation}\label{eq:sintersect}
S = B^n \cap \bigcap_{i\in I\setminus \{0,n\}} C^n_i,
\end{equation}
where $n := \max I$ and $C_i^n$ is defined by \eqref{eq:Cin}.
It follows from Lemma~\ref{lem:faces} that every face of the set $S$ is an intersection of faces of the sets that feature in the intersection \eqref{eq:sintersect}.

We first show that $B^n\cap \bd(C_j^n) \cap \bd(C_i^n) = \emptyset$ whenever $0<j<i<n$, which means that the only possible faces of $S$ are the intersections of faces of $B^n$ and $C_i^n$'s. By Lemma~\ref{lem:facedimensions} the facial dimension signature of $C_i^n$ is $\{i,n\}$ and the facial dimension signature of $B^n$ is $\{0,n\}$, hence by Lemma~\ref{lem:dimintersections} the facial dimension signature of $S$ is a subset of $I$. It will only remain to demonstrate that for any $d\in I$ there are indeed faces of dimension $d$ present in $S$. 

Let us show that indeed $B^n\cap \bd C_j^n \cap \bd C_i^n = \emptyset$, for $0<j<i<n$.  Assume the contrary, then there is some $x\in B^n\cap \bd C_j^n \cap \bd C_i^n$. Since $x\in \bd C_i^n$, we have 
\[
(x_{i+1}+c)^2 + x_{i+2}^2+\cdots + x_n^2 = r^2.
\] 
Rearranging, we have
\begin{align*}
2 x_{i+1} c &  = r^2 -  c^2 - (x^2_{i+1}+x^2_{i+2}+\cdots+x^2_{n})\\
& > (c^2+ \sqrt{2} c +1) - c^2 -1 = \sqrt{2} c,
\end{align*}
and hence $x_{i+1}>1/\sqrt{2}$. Likewise, $x_{j+1}>1/\sqrt{2}$, but then we have a contradiction
\[
1\geq \sum_{i=1}^n x_i^2 \geq x_{i+1}^2+ x_{j+1}^2 > 1.
\]

Finally, it remains to show that for any $d\in I$ there is a face of $S$ of dimension $d$. Notice that a Euclidean ball of radius $r-c$ centred at the origin is contained in $S$. Indeed, we have for any $x$ in such a ball
\[
1>(r-c)^2 \geq \sum_{i=1}^n x_i^2,
\]
hence $x\in B^n$. Also 
\[
(x_{i+1}+c)^2 + x_{i+2}^2+\cdots + x_n^2 \leq (r-c)^2 + c^2 + 2 x_{i+1} c = r^2 + 2 c( c-r  +x_{i+1}) \leq  r^2,
\]
where we used $|x_{i-1}|\leq \|x\|\leq |r-c| = r-c$.
Hence $x\in C_{i}^n$. We conclude that $S$ has a nonempty interior, and hence in itself comprises a face of dimension $n$. Since it is nonempty and compact, we also conclude that it has at least one extreme point, and so we have the zero covered as well. It remains to certify that for any other $d\in I$, $S$ has a face of this dimension. Let $d$ be one of such values, then the point $x$ with all coordinates zero except for $x_{i+1} = r-c$ is in the relative interior of such face. Indeed, since $r-c<1$, $x$ must be in the interior of $B^n$, and also 
\[
(x_{i+1}+c)^2 + x_{i+2}^2+\cdots + x_n^2  = r^2,
\]
hence $x\in \bd(C_{i}^n)$. We conclude that there must be a proper face of $C_i^n$ such that $x$ is in its relative interior, and hence its intersection with $B^n$ is a face of $S$ of the same dimension by Lemma~\ref{lem:intintersection}.

Now consider the general case when $m := \min I$ is not necessarily zero. Define the shifted set 
\[
I_0 := \{d-m\, : \, d\in I\}.
\]
We know that the set $S^0$ constructed for $I_0$ is in $\R^{n-m}$ (where $n = \max I$) and has faces of dimensions $I_0$. Let 
\[
S := S^0\oplus \R^m.
\]
This set has facial dimension signature $I$ by Lemma~\ref{lem:shift} and also has an explicit representation via $|I|-1$ quadratic inequalities,
\[
S  = \left\{x\in \R^n\, : \, 
\begin{array}{ll}
x_1^2+x_2^2+\cdots + x_{n-m}^2 \leq 1, & \\
(x_{i+1-m}+c)^2 + x_{i+2-m}^2+\cdots + x_{n-m}^2 = r^2 & \forall i \in I\setminus \{\min I,\max I\}
\end{array}
\right\}.
\]

It remains to remark that in the case when $0\in I$, the constructed set $S$ is compact: it is the intersection of closed sets one of which is compact (the ball $B^n$).

\end{proof}

\section{Efficient realisations}\label{sec:complexity}

\subsection{A better construction for complete sequences}\label{sec:optimalcomplete}
The following example shows that in the case of a complete sequence our construction can be at least exponentially worse than an optimal one. 
Let $K$ be a positive integer and consider a convex set in $\R^{2^{K}-1}$ defined by the convex quadratic inequalities
\[
x_{2^{k-1}}^2+\cdots + x_{2^k-1}^2 \leq 1, \,\,\,\, k \in \{1,\dots, K\}. 
\]
The first few inequalities are
\begin{align*}
    x_1^2& \leq 1,\\
    x_2^2+x_3^2& \leq 1,\\
    x_4^2+x_5^2+x_6^2+x_7^2& \leq 1.
\end{align*}
This set is the direct sum of convex sets defined by the solution set of each inequality. Faces of the direct sum is the direct sum of faces. Since these sets, considered as lower-dimensional sets that contribute to the direct sum, have faces of dimensions $\{0, 2^{k-1}\}$, we have to calculate 
\[
\sum_{k=1}^{K}\{0,2^{k-1}\}.
\]
We have 
\[
\{0,1\}+ \{0,2\} = \{0,1,2,3\},
\]
then 
\[
\{0,1,2,3\}+ \{0,4\} = \{0,1,2,3,4,5,6,7\},
\]
and assuming that 
\[
\sum_{k=1}^{K-1}\{0,2^{k-1}\} = \{0,1,\dots, 2^{K-1}-1\}
\]
we have 
\[
\sum_{k=1}^{K}\{0,2^{k-1}\} = \{0,1,\dots, 2^{K-1}-1\}+ \{0,2^{K-1}\} = \{0,1,\dots, 2^{K}-1\}.
\]
We have therefore constructed a convex set in the space of dimension $2^{K}-1$ with a complete dimensional sequence of faces $I = \{0,1,\dots,2^{K}-1\}$ using $K$ inequalities. We have
\[
n+1  = 2^{K},
\]
and so 
\[
K = \log_2 (n+1). 
\]

We will show in Corollary~\ref{cor:optimalcomplete} that this construction is optimal for the complete sequence, in terms of the minimum number of convex quadratic inequalities required to construct the set with the desired facial dimension signature. Together with our earlier construction, this method can also be used to represent other facial dimension signatures.

\subsection{Efficient representations and indecomposable sequences}\label{sec:integersequences}

In the previous section we used a specialised construction, where a direct sum of Euclidean balls realised a complete facial dimension signature with better complexity than an earlier construction.  A natural question is whether a similar construction can result in a more efficient realisation for other signatures. 

Generally speaking, given some convex sets $C_i\subseteq\R^{n_i}$ for $i\in \{1,\dots, m\}$, each with facial dimension signature $I_i$, using Lemma~\ref{lem:prodfaces} we conclude that the signature of the direct sum $C_1\oplus C_2\oplus \cdots \oplus C_m$ is the integer Minkowski sum 
\[
I = I_1 + I_2 + \cdots + I_m.
\]
Therefore, if a given signature is \emph{decomposable} as a sum of other integer subsets, we can use the construction of Section~\ref{sec:construction} to realise these subsets, and realise the full signature as the direct sum of these sets. Depending on the availability of decompositions, we can substantially reduce the complexity of the representations using this approach. To be more precise, a finite set of nonnegative integers $I$ is \emph{decomposable}, if there exists a pair of finite subsets $I_1$ and $I_2$ of nonnegative integers such that
\begin{itemize}
    \item 
    $I = I_1 + I_2$ and
    \item either both $I_1$ and $I_2$ have cardinality at least two,
    or one of them has cardinality at least two and the other is a singleton containing a positive integer (e.g., the set $\{0\}$ is not allowed in the decomposition).
\end{itemize}
We say that $I$ is \emph{indecomposable} if it is not decomposable.
Problems related to such decompositions of integer sequences have been studied in discrete mathematics in the research areas of additive number theory and integer sequences. For a deeper look into these areas, please see \cite{Gross2020}, \cite{Leonetti2023} and the references therein. 

\begin{theorem} If a signature $I$ can be represented as the sum $I  = I_1 + I_2 + \cdots + I_k$, where none of $I_i$ is $\{0\}$, then $I$ can be realised as the facial dimension signature of the solution set of a system of $|I_1|+|I_2|+\cdots +|I_k| - k$ convex quadratic inequalities. 
\end{theorem}
\begin{proof}
Suppose a facial dimension signature $I$ can be represented as $I  = I_1 + I_2 + \cdots + I_k$,
where each $I_i$ is a subset of nonnegative integers and none of them is $\{0\}$. Then, for each $I_i$ we apply Theorem~\ref{thm:everything} to obtain $C_i$ with facial dimension signature $I_i$ using at most $|I_i|-1$ convex quadratic inequalities. Thus, by Lemma~\ref{lem:prodfaces}, the direct sum $C:=C_1 \oplus C_2 \oplus \cdots \oplus C_k$ has the facial dimension signature $I_1 + I_2 + \cdots + I_k=I$. By construction, $C$ can be represented by using at most $\left(\sum_{i=1}^k |I_i|\right) - k$ convex quadratic inequalities as claimed.
\end{proof}

Notice that it is possible to make rational choices for both $r$ and $c$; that way, our construction can be rewritten as the solution set of a system of convex quadratic inequalities with integer coefficients.

Let us call a convex set $C \subseteq \R^n$ \emph{indecomposable} if there does not exist a linear isomorphism $A: \R^n \to \R^n$ and convex sets $C_1$, $C_2$, each with dimension at least one, such that $AC = C_1 \oplus C_2$ (otherwise, $C$ is \emph{decomposable}). By Lemma~\ref{lem:2.7}, the facial dimension signatures of $C$ and $AC$ are the same. Moreover, if $AC$ is decomposable then so is its facial dimension signature. Therefore, if $C$ is decomposable then so is its facial dimension signature. (Our results can also be adapted to convex cones.
Indecomposability of convex cones have been studied at least 50 years ago \cite{LoewySchneider1975}.) 

\subsection{Lower bounds}\label{sec:lowerbound}

The constructions discussed in the previous section can reduce the number of convex quadratic inequalities needed to represent a convex set with the desired facial dimensions. However, we do not know when this construction gives an optimal or near optimal representation. In this section, we prove a lower bound on the number of inequalities needed for a representation by a system of convex quadratic inequalities. Except for some special cases that include the complete facial dimension signature, it is an open question whether this lower bound is always achieved.

\begin{proof}[Proof of Theorem~\ref{thm:lowerbound}] Suppose that some convex set $C$ is represented as the intersection of $m$ convex quadratic inequalities. By Lemma~\ref{lem:quadraticdim} each quadratic inequality generates a convex set $C_i$ with a face of dimensions $n$ and proper faces of one fixed dimension $d_i\in \{0,\dots, n-1\}$, for $i \in \{1,\dots, m\}$. 

When we take the intersection of the solution sets of these inequalities we can only obtain faces that are the intersections of faces of these original convex quadratically defined sets, due to Lemma~\ref{lem:faces}. So every proper face of the intersection $C$ would be representable as the intersection of faces of $C_i$'s. 

Let $F$ be some face of $C$, and let $x\in \relint F$. Since $x\in C_i$ for all $i\in \{1,\dots, m\}$, we can consider $F_i = F_{\min}(x,C_i)$ for all $i$. By virtue of Lemma~\ref{lem:charminF} we have 
\[
F \subseteq \bigcap_{i\in \{1,\dots, m\}} F_i, 
\]
and since there are no smaller faces of $C_i$'s that would contain $F$, we conclude that the above inclusion is actually an equality. 

In a small neighbourhood of $x$ (that can be assumed to be located at the origin) the faces $F_i$, as well as the face $F$ coincide with linear subspaces whose dimensions are the dimensions of these faces. Hence the possible dimensions that $F$ can have are exactly the dimensions of the linear subspaces that can be obtained by intersecting $m$ linear subspaces of the dimensions of the faces $F_i$. 

Since every face is either a proper face or the set itself that has dimension $n$, the dimension of each intersection is determined by the dimension of the relative intersection of proper faces of $C_j$'s with $j\in J$, with $J\subseteq \{1,\dots, m\}$ (since faces of dimension $n$ do not affect the dimension of the intersection).

For any subset $J$ of $\{1,\dots, m\}$ the possible dimensions that we can obtain by intersecting \emph{all} of the proper faces of the sets $C_j$ with $j\in J$, are 
\[
\max\left\{0,\sum_{j\in J} d_j - (|J|-1) n \right\}  \leq d\leq \min_{j\in J} d_j,
\]
obtained from checking the possible dimensions of linear subspaces obtained by intersecting the linear subspaces of dimensions $d_j$ with $j\in J$. 

Therefore, the only possible dimensions of faces of $C$ are 
\[
I = \{n\}\cup \bigcup_{J\subseteq \{1,\dots,m\}} I_J,
\]
where 
\[
I_J := \left[\max\left\{0,\sum_{j\in J} d_j- (|J|-1) n \right\} ,\min_{j\in J} d_j\right].
\]

Now notice that if $J'\subseteq J\subseteq\{1,\dots,m\}$ and $\min_{j\in J} d_j = \min_{j\in J'} d_j$, then  
\[
\sum_{j\in J} d_j- (|J|-1) n = \sum_{j\in J'} d_j- (|J'|-1) n 
+ \sum_{j\in J\setminus J'} (d_j - n).
\]
Since the last term is negative, we conclude that $I_{J'}\subset I_{J}$, and hence if we make an assumption that the sequence $(d_i)_{i=1}^m$ is descending, then 
\[
I = \{n\}\cup \bigcup_{j=1}^{m} \left[\max\left\{0,\sum_{i=1}^j d_i- (j-1) n \right\} , d_j \right]
\]
is a set which contains all possible dimensions of faces of the set $C$. 
We conclude that if we want to represent a dimensional sequence via a convex set $C$ it cannot be represented more efficiently (in terms of the number of inequalities) than the minimum number needed for the representation \eqref{eq:intervals}.
\end{proof}

For a given nonnegative integer sequence $I$, and a given positive integer $k$, we can check whether $k$ is a lower bound given by Theorem~\ref{thm:lowerbound}, by solving an integer programming problem.

\begin{corollary}\label{cor:optimalcomplete}
If $I$ is a complete sequence, that is, $I = [\min I,\,\max I]\cap \Z_+$, then the minimum number of convex quadratic inequalities needed to represent this sequence as the facial dimension signature of some convex set is $\lceil\log_2(\max I-\min I +1)\rceil$.
\end{corollary}
\begin{proof} We will use the bound of Theorem~\ref{thm:lowerbound} to prove this result. Let $n = \max I$.  The sequence of intervals that comprise the bound  monotone (for larger $j$ both lower and upper bounds are larger), and hence to ensure that we have no gaps in this representation, we have to have $d_1 = n-1$, for any $k\in \{1,\dots, m\}$
\[
d_1+d_2+\cdots + d_k - (k-1)n  \leq d_{k+1}+1,
\]
and
\[
d_1+d_2+\cdots + d_m - (m-1)n \leq  \min I. 
\]
Hence  
\[
d_{k+1}\geq d_1+d_2+\cdots + d_k - (k-1)n-1.
\]
Suppose that we have some sequence that already satisfies these inequalities. Then if there is some $k<n-1$ such that 
\[
d_{k+1}> d_1+d_2+\cdots + d_k - (k-1)n-1,
\]
we can reduce $d_{k+1},d_{k+2},\dots, d_m$ by one and still obtain a valid sequence (discarding any numbers that are smaller than $\min I$). Moreover, if there are any duplicate numbers $d_j=d_{j+1}$ in the sequence, we must have 
\[
d_{j+1} = d_j \geq  d_1 + d_2+\cdots + d_{j-1} - (j-2)n-1 > d_1+d_2+\cdots + d_{j} - (j-1)n-1, 
\]
hence the resulting reduced sequence must have no duplicates. 
We can therefore assume that for an optimal sequence of dimensions we must have 
\[
d_{k+1}= d_1+d_2+\cdots + d_k - (k-1)n-1 = \sum_{i=1}^k (d_i-n) + n-1 = d_k + (d_k-n) = 2 d_k-n.
\]
Since we must have $d_1 = n-1$, we then have 
\[
d_2 = 2 (n-1) -n = n-2, \; d_3 = 2 d_{m-2}-n = 2 (n-2)-n = n-4,
\]
and generally it is easy to see by induction that 
\[
d_k = n - 2^{k-1}.
\]
We must have for $d_m$ 
\[
d_1+d_2+\cdots + d_m - (m-1)n \leq  \min I,
\]
equivalently 
\[
\min I \geq n+ \sum_{i=1}^m (d_i-n) = n- \sum_{i=1}^m 2^{i-1} = n +1 -2^{m},
\]
so we must have 
\[
2^m \geq n+1-\min I,
\]
or 
\[
m \geq \log_2 (n+1-\min I).
\]
\end{proof}

\section{Conclusions and open problems}\label{sec:conclusions}
Our constructions are easily extended to the setting of convex cones.
We proved that our lower bound given in Theorem~\ref{thm:lowerbound} is tight (e.g., for complete sequences). However, we do not know in general, how loose the strongest of the bounds given in that theorem can be. We also do not know if considering more general objects than convex quadratic inequalities, for instance, general spectrahedra, spectrahedral shadows or hyperbolicity cones, may give us more efficient constructions (in some appropriate sense). We are also curious about the relation between the facial dimension signature of a convex set and its polar, and whether it is always possible to realise facial dimension signatures with convex sets that are facially exposed along with their polars. Next, we discuss these and related questions in more detail.

\renewcommand{\theenumi}{\arabic{enumi}} 

\subsection{Some indecomposable sequences}

\begin{enumerate}
    \item Can the lower bound given in Theorem~\ref{thm:lowerbound} be improved for some interesting family of signatures? In particular, consider the set of prime integers up to and including the $k$th prime, $I:=\{0, 2, 3, \ldots, p_k\}$ together with zero. The cardinality of $I$ is $k+1$ and $p_k \approx k\ln k$. This set of integers is indecomposable. Our construction leads to a representation of a convex set with this prime sequence as its facial dimension signature using $k$ convex quadratic inequalities. Is there a better construction with convex quadratic inequalities for this family of facial dimension signatures? 
    
    \item For each positive integer $n$, what is the minimum number of convex quadratic inequalities whose solution set has the facial dimension signature
    \[
\{0,1,4,9,16, \ldots, n^2\} ?
    \]
 This set of integers is also indecomposable and coincides with the facial dimension signature of the cone of $n$-by-$n$ Hermitian positive semidefinite matrices over the complex numbers.

\item For nonnegative integers $n$, let us define $t(n):= \frac{n(n+1)}{2}.$
Then the facial dimension signature of the cone of $n$-by-$n$ symmetric
positive semidefinite matrices is
\[
\left\{0, 1, 3, \ldots, t(n) \right\} = \left\{t(k) : k \in \{0, 1, \ldots, n\}\right\}. 
\]
This set of integers is also indecomposable. What is the minimum number of convex quadratic inequalities whose solution set has this facial dimension signature?

\end{enumerate}

\subsection{Representational questions}

For the efficiency of representations, so far in this paper, we only focused on efficiency in terms of the minimum number of convex quadratic inequalities that can represent a convex set with the given facial dimension signature. However, there are many other approaches to complexity and efficiency of representations that can be explored. Some of these related questions can be more important for many optimisation applications. 

For instance, we can ask what is the smallest spectrahedral representation of a closed convex set $C$ with the given dimensional signature, that is, what is the smallest $d$ such that 
\[
C = \left\{x \in\R^n \, :\, \sum_{i=1}^n A_n x_n + A_0 \succeq 0 \right\}, 
\] where $A_i \in \Sym^d$, for $i\in \{0, \dots, n\}$. 

Likewise, we can consider the smallest spectrahedral shadows and smallest hyperbolicity cone representations. Moreover, we can define the \emph{size} of the representation in different ways. For example, in the above we only mentioned $d$, the number of rows/columns of the matrices in the representation, one may consider demanding that all entries be rational, and then define the $\emph{size}$ also depending on the bit size of the data $A_0, A_1, \ldots, A_n$.

\begin{enumerate}

    \item What are the sizes of smallest spectrahedral representations for a given facial dimension signature?
    \item What are the sizes of optimal realisations by hyperbolicity cones?
    \item What are the sizes of optimal realisations by spectrahedral shadows? 
\end{enumerate}

\subsection{Dimensions of unions and measured facial dimension signatures}

We can consider more sophisticated facial signatures by defining a \emph{measured facial dimension signature} of a convex set $C$, as a pair of finite sets one of which is the facial dimension signature $I$, and the other is an ordered finite set $M$ of nonnegative real numbers where $|M| = |I|$ such that for each $i \in I$, the corresponding entry $M_i$ of $M$ is the Hausdorff dimension of the union of all faces of $C$ with dimension $i$. For example, for the convex sets in Fig.~\ref{fig:constIllustrate}, we have (respectively)
    \[
I= \{0,2,3\}, M=(2, 2, 3) \hspace{1cm} \mbox{ and }
\hspace{1cm} I= \{0,1,3\}, M=(2, 2, 3) \]
\[
\mbox{ and } \hspace{0.5cm} I= \{0,1,2,3\}, M=(2, 2, 2, 3).
\]
As a family of examples of arbitrarily large dimension, for any positive integer $n \geq 3$, consider the $n$-dimensional second-order cone. Its measured facial dimension signature is $I=\{0,1,n\}$, $M=(0, n-1, n)$. (The union of one-dimensional faces of this cone is equal to its boundary, and the Hausdorff dimension of the boundary is $(n-1)$.)
One can construct yet more complicated convex sets where the
Hausdorff dimension of the union of faces of certain dimension is not an integer. For some interesting examples, see \cite{RoshYost,FractalCurvature}.

We can then ask our realisability questions for given pairs $(I,M)$ or for given triples $(I,L,U)$ where $L$ and $U$ are ordered sets of nonnegative real numbers.
\begin{enumerate}
\item
Given a finite set of nonnegative integers $I$ what are the \emph{best} lower and upper bounds $L$ and $U$ such that for every convex set with facial dimension signature $I$, its measured facial dimension signature $(I,M)$ satisfies $L \leq M\leq U$? (Here, \emph{best} $L$ means, there does not exist $L'\geq L$ with $L' \neq L$ satisfying the same condition, analogously for $U$.)

More ambitiously, we may aim for sharper answers as follows.

 \item
Given a finite set of nonnegative integers $I$, characterize the set of $M$ such that there exists a convex set with measured facial dimension signature $(I,M)$.
\end{enumerate}

\subsection{Questions related to polars}\label{sec:polars}

For every set $C$ in $\R^n$, we define its \emph{polar} as
\[
C^{\circ} := \left\{s \in \R^n \, : \, \iprod{x}{s} \leq 1 \,\, \forall x \in C \right\}.
\]
Our first open question involving the polars is:

\begin{enumerate}

    \item Can we realise any  facial dimension signature $I$  with some facially exposed compact convex set such that its polar is also facially exposed?

\end{enumerate}

The above question does not have any restrictions on the facial dimension signature of the polar.
Next, suppose we are given two facial dimension signatures $I$ and $J$ such that $\min I = \min J = 0$ and $\max I = \max J$. In the two-dimensional case, there are only three possible pairs of signatures: $\{(0,2),(0,2)\}$, $\{(0,1,2),(0,1,2)\}$ and $\{(0,2),(0,1,2)\}$ (we can interchange the primal and the polar to obtain the permuted pair of signatures). The first two cases are realisable by the disk and with a triangle respectively, and for the last case we can consider the intersection of two regions defined by nondegenerate convex quadratic inequalities, as shown in Fig.~\ref{fig:pd2d}.
\begin{figure}[ht]
    \centering
    \begin{overpic}[
    width=0.4\textwidth]{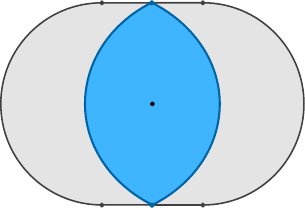}
    \put(53,28){$0$}
    \put(40,44){primal}
    \put(6,27){polar}
    \end{overpic}
    \caption{A pair of primal and polar compact convex sets that realise the pair of facial dimension signatures $(0,2)$ and $(0,1,2)$.}
    \label{fig:pd2d}
\end{figure}
One way to go about constructing primal and dual pairs of compact convex sets with desired pair of signatures $(I,J)$ is to arrange faces with signature $\{0,1,2\}\cap I$ smoothly on the boundary of a compact three-dimensional convex set, that way ensuring that the primal set has the signature $I$, while the polar has signature $\{0,3\}$. The same construction can be repeated locally on the polar, changing the facial dimension signature of the polar to $J$, but making no impact on the primal signature. This idea is illustrated in Figs.~\ref{fig:ppexample}, where we show an example of a primal compact convex set with signature $(0,1,3)$ and its polar with signature $(0,3)$. 
\begin{figure}[ht]
    \centering
    \includegraphics[height=0.22\linewidth]{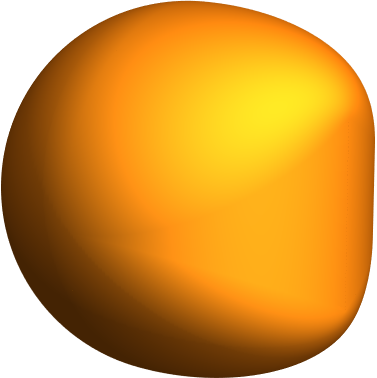}\;
    \includegraphics[height=0.22\linewidth]{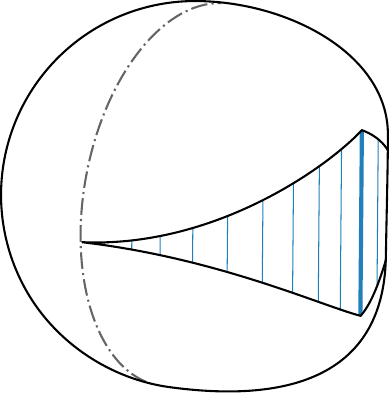}\qquad \quad
    \includegraphics[height=0.22\linewidth]{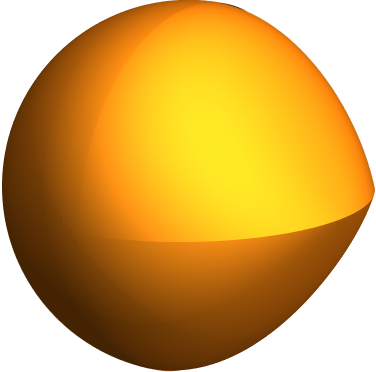}\;
        \includegraphics[height=0.22\linewidth]{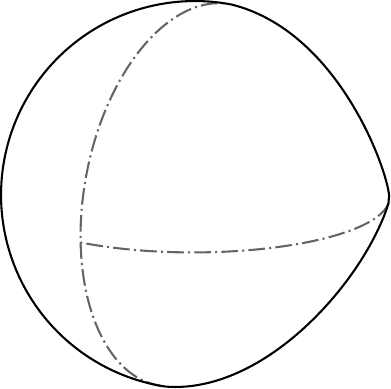}
    \caption{An example of a primal and polar pair of compact convex sets with facial dimension signatures $(0,1,3)$ and $(0,3)$ respectively (the visible vertical equator on the polar shown in the right-hand side is an artifact of Mathematica rendering: the surface is smooth along that circle).}
    \label{fig:ppexample}
\end{figure}

\begin{enumerate}

\setcounter{enumi}{1}

    \item Suppose that we are given two facial dimension sequences, $I$ and $J$ and for the sake of simplicity, suppose that $\min I = \min J = 0$ and $\max I = \max J$. Is it true that for any such pair $I,J$ there exists a compact convex set $C$ such that its facial dimension signature is $I$, and the facial dimension signature of its polar $C^\circ$ is $J$?

    It is clear from our earlier discussion in this subsection that in the case of $\R^3$ and in lower dimensional cases, all pairs of facial dimension signatures are possible. It seems conceivable that these ideas we outlined can be generalised to higher dimensions.

    \item In the previous subsections of this section, we posed our open problems over the primal set. For each of these questions, it would be interesting to investigate suitable primal-dual versions (in addition to the requirements on the primal set, requiring that the polar set also possess related, suitable desired properties).
    
\end{enumerate}

\bibliographystyle{plain}
\bibliography{refs}

\end{document}